\documentclass[11pt,a4paper]{article}
\usepackage[top=3cm, bottom=3cm, left=3cm, right=3cm]{geometry}
\usepackage[latin1]{inputenc}
\usepackage{csquotes}
\usepackage{amsmath}
\usepackage{amsthm}
\usepackage{amsfonts}
\usepackage{amssymb}
\usepackage{graphics}
\usepackage{float}
\usepackage[hang,flushmargin]{footmisc} %nonindent footnote

\usepackage{amsmath,amsthm,amssymb,graphicx, multicol, array}
\usepackage{enumerate}
\usepackage{enumitem}
\usepackage{xcolor}
\usepackage{multicol}

\usepackage[pagebackref,bookmarks,colorlinks,breaklinks,linktoc=page]{hyperref}
\hypersetup{linkcolor=blue,citecolor=red,filecolor=blue,urlcolor=blue} 

\usepackage{tabto}

\numberwithin{equation}{section}

\usepackage{tikz}
\usepackage{pgfplots}

\usetikzlibrary{matrix}
\usepackage{mathrsfs}

\DeclareMathOperator{\ord}{ord}

\newtheorem{theorem}{Theorem}[section]
\newtheorem{lemma}{Lemma}[section]

\newtheorem{rmk}{Remark}[section]

\newcommand{\N}{\mathbb{N}}

\newcommand{\C}{\mathbb{C}}
\newcommand{\F}{\mathbb{F}}
\newcommand{\tP}{\mathbb{P}}

\usepackage{fancyhdr}
\title{Small Prime Primitive Roots in Arithmetic Progressions}
\date{}
\author{N. A. Carella}

\begin{document}
\maketitle

%%%%%%%%%%%%%%%%%%%%%%%%%%%%%%%%%%%%%%%%%%%%%%%%%%%%%%%%%%%%%%%%%

\begin{abstract}
Let $p>1$ be a large prime number, let $q=O(\log\log p)$ and let $1\leq a<q$ be a pair of relatively prime integers. It is proved that there is a prime primitive root $u\ll (\log p)(\log \log p)^{4+\varepsilon}$ such that $u\equiv a\bmod q$ in the prime finite field $\mathbb{F}_p$. \let\thefootnote\relax\footnote{ \today \date{} 
	}
\end{abstract}
%\tableofcontents
%SSSSSSSSSSSSSSSSSSSSSSSSSSSSSSSSSSSSSSSSSSSSSSSSSSSSSSSSSSSSSSSS
%SSSSSSSSSSSSSSSSSSSSSSSSSSSSSSSSSSSSSSSSSSSSSSSSSSSSSSSSSSSSSSSS
%SSSSSSSSSSSSSSSSSSSSSSSSSSSSSSSSSSSSSSSSSSSSSSSSSSSSSSSSSSSSSSSS
%SSSSSSSSSSSSSSSSSSSSSSSSSSSSSSSSSSSSSSSSSSSSSSSSSSSSSSSSSSSSSSSS
%SSSSSSSSSSSSSSSSSSSSSSSSSSSSSSSSSSSSSSSSSSSSSSSSSSSSSSSSSSSSSSSS
%SSSSSSSSSSSSSSSSSSSSSSSSSSSSSSSSSSSSSSSSSSSSSSSSSSSSSSSSSSSSSSSS
\section{Introduction }\label{S9955B}%\hypertarget{S9955B}
The unconditional upper bounds of the least primitive root $u\ne\pm1,v^2$ in the prime finite field $\F_p$ seem to be exponential $u\ll p^{1/4+\varepsilon}$, see {\color{red}\cite[Theorem 3]{BD1962}}, and the conditional upper bounds are of the forms $u\ll (\log p)^{6+\varepsilon}$, see {\color{red}\cite[Theorem 1.3]{SV1992}}. Moreover, the heuristic for the smaller upper bound $u\ll (\log p)(\log\log p)^{2}$ appears in {\color{red}\cite[Section 4]{BE1997}}. This note proposes a new result on the theory of primitive roots in finite fields. This unconditional result is true for all large primes and breaks the upper bound exponential barrier. 

\begin{theorem} \label{thm9955P.800}\hypertarget{thm9955P.800} Let $p>1$ be a large prime number, let $q=O(\log\log p)$ and let $1\leq a<q$ be a pair of relatively prime integers. Then there exists a prime primitive root $u\ne \pm1, v^2$ in the prime finite field $\mathbb{F}_p$ such that
	\begin{multicols}{2}	
	\begin{enumerate}[font=\normalfont, label=(\roman*)]
		\item $\displaystyle u\ll (\log p)(\log \log p)^{4+\varepsilon},$
		\item $\displaystyle u\equiv a\bmod q,$
	\end{enumerate}
\end{multicols}	
here $\varepsilon>0$ is a small number. In particular, the number of primitive roots on the arithmetic progression has the lower bound
	\begin{equation} \label{eq9955P.800v}
		N_0(x,q,a)\gg  (\log p)(\log \log p)^{2+\varepsilon}
	\end{equation} 
as $p\to\infty$.	
\end{theorem}	
\vskip .1 in
The best and closest result in this direction is the existence of a subset of prime primitive roots 
\begin{equation}\label{eq9955P.800c}
	\mathcal{R}=\{g^*(p)\leq H\}
\end{equation} of cardinality
\begin{equation}\label{eq9955P.800d}
	\#\mathcal{R}=\frac{\varphi(p-1)}{p-1}\pi(H)\left(1+O\left( \frac{1}{(\log H)^B}\right)  \right),
\end{equation}
where $H>e^{c(\log\log p)(\log\log\log p)}$ for all primes $p$ but a subset of primes of zero density, see {\color{red}\cite[Theorem 1]{EP1969}}, this asymptotic formula is quite similar to the lower bound in \eqref{eq9955P.800v}.\\

This innovation is made possible by a new characteristic function for primitive root described in \hyperlink{S9955D}{Section} \ref{S9955D}. The proof of \hyperlink{thm9955P.800}{Theorem} \ref{thm9955P.800} appears in \hyperlink{S9955P}{Section} \ref{S9955P}. Given the prime factorization of the totient $p-1$ an easy application of this result leads to a polynomial time algorithm for determining primitive root in the prime finite field $\F_p$. The most recent algorithm for searching for primitive roots and a survey of the literature appears in \cite{SI2018}. An explicit lower bound of the prime $p>p_0$ is computed in \hyperlink{S9955PT}{Section} \ref{S9955PT}.

%SSSSSSSSSSSSSSSSSSSSSSSSSSSSSSSSSSSSSSSSSSSSSSSSSSSSSSSSSSSSSSSS
%SSSSSSSSSSSSSSSSSSSSSSSSSSSSSSSSSSSSSSSSSSSSSSSSSSSSSSSSSSSSSSSS
%SSSSSSSSSSSSSSSSSSSSSSSSSSSSSSSSSSSSSSSSSSSSSSSSSSSSSSSSSSSSSSSS
%SSSSSSSSSSSSSSSSSSSSSSSSSSSSSSSSSSSSSSSSSSSSSSSSSSSSSSSSSSSSSSSS
%SSSSSSSSSSSSSSSSSSSSSSSSSSSSSSSSSSSSSSSSSSSSSSSSSSSSSSSSSSSSSSSS
%SSSSSSSSSSSSSSSSSSSSSSSSSSSSSSSSSSSSSSSSSSSSSSSSSSSSSSSSSSSSSSSS
\section{Representations of the Characteristic Function}\label{S9955D}\hypertarget{S9955D}
The \textit{multiplicative order} of an element in a finite field $\mathbb{F}_p$ is defined by $\ord_p u=\min\{k:u^k\equiv 1 \bmod p\}$. An element $u\ne \pm1,v^2$ is called a primitive root if $\ord_p u=p-1.$ The characteristic function \(\Psi :G\longrightarrow \{ 0, 1 \}\) of primitive elements is one of the standard analytic tools employed to investigate the various properties of primitive roots in cyclic groups \(G\). Many equivalent representations of the characteristic function $\Psi $ of primitive elements
are possible, a few are investigated here. 
%MMMMMMMMMMMMMMMMMMMMMMMMMMMMMMMMMMMMMMMMMMMMMMMMMMMMMMMMMMMM
%MMMMMMMMMMMMMMMMMMMMMMMMMMMMMMMMMMMMMMMMMMMMMMMMMMMMMMMMMMMM
\subsection{Divisor Dependent Characteristic Function}		
The divisor dependent characteristic function was developed about a century ago, see {\color{red}\cite[Theorem 496]{LE1927}}
, {\color{red}\cite[p.\; 258]{LN1997}}, et alia. This characteristic function detects the order of an element by means of the divisors of the totient $p-1$. The precise description is stated below.

\begin{lemma} \label{lem9955.200D} \hypertarget{lem9955.200D} Let \(p\geq 2\) be a prime and let \(\chi\) be a multiplicative character of order $\ord  \chi =d$. If \(u\in
	\mathbb{F}_p\) is a nonzero element, then
	\begin{equation}
		\Psi (u)=\frac{\varphi(p-1)}{p-1}\sum _{d\mid p-1} \frac{\mu(d)}{\varphi(q)}\sum _{\ord \chi =d} \chi(u)
		=\left \{
		\begin{array}{ll}
			1 & \text{ if } \ord_p (u)=p-1,  \\
			0 & \text{ if } \ord_p (u)\neq p-1, \\
		\end{array} \right .\nonumber
	\end{equation}
	where $\mu:\N\longrightarrow \{-1,0,1\}$ is the Mobius function.
\end{lemma}	

%MMMMMMMMMMMMMMMMMMMMMMMMMMMMMMMMMMMMMMMMMMMMMMMMMMMMMMMMMMMM
%MMMMMMMMMMMMMMMMMMMMMMMMMMMMMMMMMMMMMMMMMMMMMMMMMMMMMMMMMMMM
\subsection{Divisorfree Characteristic Function}	
A new \textit{divisors-free} representation of the characteristic function of primitive element is developed here. It detects the order \(\text{ord}_p
(u) \geq 1\) of the element \(u\in \mathbb{F}_p\) by means of the solutions of the equation \(\tau ^n-u=0\) in \(\mathbb{F}_p\), where
\(u,\tau\) are constants, and $n$ is a variable such that \(1\leq n<p-1, \gcd (n,p-1)=1\). 

\begin{lemma} \label{lem9955.200A} \hypertarget{lem9955.200A} Let \(p\geq 2\) be a prime and let \(\tau\) be a primitive root mod \(p\) and  let \(\psi \neq 1\) be a nonprincipal additive character of order $\ord  \psi =p$. If \(u\in
	\mathbb{F}_p\) is a nonzero element, then
	\begin{equation}
		\Psi (u)=\sum _{\gcd (n,p-1)=1} \frac{1}{p}\sum _{0\leq s\leq p-1} \psi \left ((\tau ^n-u)s\right)
		=\left \{
		\begin{array}{ll}
			1 & \text{ if } \ord_p (u)=p-1,  \\
			0 & \text{ if } \ord_p (u)\neq p-1. \\
		\end{array} \right .\nonumber
	\end{equation}
\end{lemma}	
\begin{proof}[\textbf{Proof}] Set the additive character $\psi(s) =e^{i 2\pi  as/p}\in \C$. As the index $n\in \mathscr{R}=\{n<p:\gcd(n,p-1)=1\}$ ranges over the integers relatively prime to $\varphi(p-1)=p-1$, the element $\tau ^n\in \F_p ^{\times}$ ranges over the primitive roots
	modulo $p$. Accordingly, the equation $a=\tau ^n- u=0$ has a unique solution $n\geq1$ if and only if the fixed element $u\in \F_p$ is a primitive root. This implies that the inner sum in 	
	\begin{equation}\label{eq9977FF.300DF}
		\sum_{\gcd (n,p-1)=1} \frac{1}{p}\sum _{0\leq s< p} e^{i 2\pi \frac{(\tau ^n-u)s}{p}}=
		\left \{\begin{array}{ll}
			1 & \text{ if } \ord_{p} (u)=p-1,  \\
			0 & \text{ if } \ord_{p} (u)\ne p-1. \\
		\end{array} \right.
	\end{equation} 
	collapses to $\sum _{0\leq s< p} e^{i 2\pi as/p}=\sum _{0\leq s< p} 1=p $. Otherwise, if the element $u\in \F_p$ is not a primitive root, then the equation $a=\tau ^n- u=0$ has no solution $n\geq1$, and the inner sum in \eqref{eq9977FF.300DF} collapses to $\sum _{0\leq s< p} e^{i 2\pi as/p}=0$,
	this follows from the geometric series formula $\sum_{0\leq n\leq  N-1} w^n =(w^N-1)/(w-1)$, where $w=e^{i 2\pi a/p}\ne1$ and $N=p$. 
	This completes the verification.	 
\end{proof}

Another approach to the construction of divisorfree characteristic function is via the discrete logarithm $\log_{\tau}:\F_p^{\times}\longrightarrow \F_p$ in finite field $\F_p$.

\begin{lemma} \label{lem9955FF.200DFL} \hypertarget{lem9955FF.200DFL}  Let \(p\geq 2\) be a prime and let \(\tau\) be a primitive root mod \(p\) and  let \(\psi \neq 1\) be a nonprincipal additive character of order $\ord  \psi =p$. If \(u\in
	\mathbb{F}_p\) is a nonzero element, then
	\begin{equation}
		\Psi (u)=\sum _{\substack{1\leq s< p-1\\\gcd (s,p-1)=1}} \frac{1}{p}\sum _{0\leq t< p} \psi \left ((s-\log_{\tau}u)t\right)
		=\left \{
		\begin{array}{ll}
			1 & \text{   \normalfont if } \ord_p (u)=p-1,  \\
			0 & \text{   \normalfont if } \ord_p (u)\neq p-1. \\
		\end{array} \right .\nonumber
	\end{equation}
\end{lemma}	

\begin{proof}[\textbf{Proof}] Let $\mathscr{R}=\{s<p:\gcd(s,p-1)=1\}$ and set the additive character $\psi(s) =e^{i 2\pi  as/p}\in \C$. Consider a fixed element $u=\tau^s\in \F_p$ and its discrete logarithm $\log_{\tau}u=s$. As the index varies over the set $s\in \mathscr{R}$ ranges over the integers relatively prime to $\varphi(p)=p-1$, the equation 
	\begin{equation}\label{eq9977FF.300DFL1}
		a=s- \log_{\tau}u=0
	\end{equation} 
	has a unique solution $s=\log_{\tau}u$ if and only if the fixed element $u\in \F_p$ is a primitive root. This implies that the inner sum in 	
	\begin{equation}\label{eq9977FF.300DFL2}
		\sum _{\substack{1\leq s< p-1\\\gcd (s,p-1)=1}} \frac{1}{p}\sum _{0\leq t< p} e^{i 2\pi \frac{(s-\log_{\tau}u)t}{p}}=
		\left \{\begin{array}{ll}
			1 & \text{   \normalfont if } \ord_{p} (u)=p-1,  \\
			0 & \text{   \normalfont if } \ord_{p} (u)\ne p-1. \\
		\end{array} \right.
	\end{equation} 
	collapses to $\sum _{0\leq s< p} e^{i 2\pi as/p}=\sum _{0\leq s< p-1} 1=p $. Otherwise, if the element $u-1\in \F_p$ is not a primitive root, then the equation \eqref{eq9977FF.300DFL1} has no solution $s\in \mathscr{R}$ , and the inner sum in \eqref{eq9977FF.300DFL2} collapses to $\sum _{0\leq s< p-1} e^{i 2\pi as/p}=0$,
	this follows from the geometric series formula $\sum_{0\leq n\leq  N-1} w^n =(w^N-1)/(w-1)$, where $w=e^{i 2\pi a/p}\ne1$ and $N=p$. 
	This completes the verification.	 
\end{proof}

%RRRRRRRRRRRRRRRRRRRRRRRRRRRRRRRRRRRRRRRRRRRRRRRRRRRRRRRRRRRRRRRRRRRRRRRRRR
%RRRRRRRRRRRRRRRRRRRRRRRRRRRRRRRRRRRRRRRRRRRRRRRRRRRRRRRRRRRRRRRRRRRRRRRRRR
%RRRRRRRRRRRRRRRRRRRRRRRRRRRRRRRRRRRRRRRRRRRRRRRRRRRRRRRRRRRRRRRRRRRRRRRRRR
%RRRRRRRRRRRRRRRRRRRRRRRRRRRRRRRRRRRRRRRRRRRRRRRRRRRRRRRRRRRRRRRRRRRRRRRRRR
\section{Finite Summation Kernel and Gauss Sum}
An upper bound for some elementary exponential sums are provided in this section. 
%RRRRRRRRRRRRRRRRRRRRRRRRRRRRRRRRRRRRRRRRRRRRRRRRRRRRRRRRRRRRRRRRRRRRRRRRRR
\subsection{Finite Summation Kernel}
\begin{lemma}   \label{lem5555.400B}\hypertarget{lem5555.400B}  Let \(p\geq 2\) be large prime, and let $\omega=e^{i2 \pi/p} $ be a $p$th root of unity. Then,
\begin{equation}
		\sum_{1 \leq t\leq p-1}\Bigg | \sum_{\substack{1\leq n\leq p-1\\\gcd(n,p-1)=1}} \omega^{tn}  \Bigg  |\ll  p^{1+\delta} \log p ,
\end{equation}		where $\delta>0$ is a small number. 
\end{lemma} 

\begin{proof}[\textbf{Proof}] Use the inclusion exclusion principle to rewrite the exponential sum as
	\begin{eqnarray} \label{eq5555.400i}
		\sum_{\substack{1\leq n\leq p-1\\\gcd(n,p-1)=1}}\omega^{tn}&=& \sum_{n \leq p-1} \omega^{tn}  \sum_{\substack{d \mid p-1 \\ d \mid n}}\mu(d)  \nonumber \\
		&=& \sum_{d \mid p-1} \mu(d) \sum_{\substack{n \leq p-1 \\ d \mid n}} \omega^{tn}\nonumber \\
		& =&\sum_{d\mid p-1} \mu(d) \sum_{m \leq (p-1)/ d} \omega^{dtm} \\
		&=& \sum_{d \mid p-1} \mu(d) \frac{\omega^{dt}-\omega^{dt((p-1)/d+1)}}{1-\omega^{dt}} \nonumber.
	\end{eqnarray} 
	Now, the parameters are $p$ prime, $\omega=e^{i2 \pi/p}$, the integers $t \in [1, p-1]$ and $d \leq p-1<p$. This data implies that $\pi dt/p\ne k \pi $ with $k \in \mathbb{Z}$, so the sine function $\sin(\pi dt/p)\ne 0$ is well defined. Consequently, the absolute value satisfies
	\begin{equation}
		\left |\frac{\omega^{dt}-\omega^{dt((p-1)/d+1)}}{1-\omega^{dt}} \right |\leq 	\left | \frac{2}{\sin( \pi dt/ p)} \right |.
	\end{equation}
For each $d\mid p-1$, the map $t\longrightarrow z\equiv dt \bmod p$ is a permutation in the finite field $\F_p$. Thus, using standard manipulations, and $z/2 \leq \sin(z) <z$ for $0<|z|<\pi/2$, the last expression becomes
	\begin{eqnarray}
		\sum_{1 \leq t\leq p-1}\Bigg | 	\sum_{\substack{1\leq n\leq p-1\\\gcd(n,p-1)=1}}\omega^{tn}\Bigg  |&\leq&\sum_{d \mid p-1,} \sum_{1 \leq t\leq p-1}	\left | \frac{2}{\sin( \pi dt/ p)} \right |\\
		&\leq&\sum_{d \mid p-1,} \sum_{1 \leq z\leq p-1}	 \frac{2p}{ \pi z} \nonumber\\
		&\ll&  p^{1+\delta} \log p \nonumber,
	\end{eqnarray}
	where $\sum_{d \mid p-1}1=d(p-1)\ll p^{\delta}$ is the number of divisor in $p-1$ and $\delta>0$ is a small number. 
\end{proof}

%RRRRRRRRRRRRRRRRRRRRRRRRRRRRRRRRRRRRRRRRRRRRRRRRRRRRRRRRRRRRRRRRRRRRRRRRRR
\subsection{Gauss-type Sum}
An elementary exponential sum estimate of a Gauss-type sum is provided in this section. 
\begin{lemma}   \label{lem1234A.150A}\hypertarget{lem1234A.150A}  
 Let \(p\geq 2\) be a prime, let $\chi(t)=e^{i2 \pi t/p} $ and  $\psi(t)=e^{i2\pi  \tau^t/p}$ be a pair of characters. Then, the Gaussian sum has the upper bound
	\begin{equation} \label{eq3-355}
		\left |\sum_{1 \leq t \leq p-1}    \chi(t) \psi(t) \right | \leq 4 p^{1/2} \log p.\nonumber
	\end{equation}
\end{lemma} 

\begin{proof}[\textbf{Proof}] This Guassian type sum can be rewritten as
	\begin{equation} \label{eq3-355b}
\sum_{1 \leq t \leq p-1}    \chi(t) \psi(t)=\sum_{1 \leq t \leq p-1}   e^{\frac{i\pi (at+b\tau^t)}{p}},
\end{equation}	
where $a,b\ne0$ are constants, has an upper bound completely independent of $a$ and $b$. 	
	The original proof using Legendre resolvent seems to be that in \cite{ML1972}, a second proof was recent given in \cite{CC2009}.
\end{proof}
%SSSSSSSSSSSSSSSSSSSSSSSSSSSSSSSSSSSSSSSSSSSSSSSSSSSSSSSSSSSSSSSSSSSSSSSSSSSSSSSSSSS
%SSSSSSSSSSSSSSSSSSSSSSSSSSSSSSSSSSSSSSSSSSSSSSSSSSSSSSSSSSSSSSSSSSSSSSSSSSSSSSSSSSS
%SSSSSSSSSSSSSSSSSSSSSSSSSSSSSSSSSSSSSSSSSSSSSSSSSSSSSSSSSSSSSSSSSSSSSSSSSSSSSSSSSSS
%SSSSSSSSSSSSSSSSSSSSSSSSSSSSSSSSSSSSSSSSSSSSSSSSSSSSSSSSSSSSSSSSSSSSSSSSSSSSSSSSSSS
%SSSSSSSSSSSSSSSSSSSSSSSSSSSSSSSSSSSSSSSSSSSSSSSSSSSSSSSSSSSSSSSSSSSSSSSSSSSSSSSSSSS
%SSSSSSSSSSSSSSSSSSSSSSSSSSSSSSSSSSSSSSSSSSSSSSSSSSSSSSSSSSSSSSSSSSSSSSSSSSSSSSSSSSS
%SSSSSSSSSSSSSSSSSSSSSSSSSSSSSSSSSSSSSSSSSSSSSSSSSSSSSSSSSSSSSSSSSSSSSSSSSSSSSSSSSSS
%RRRRRRRRRRRRRRRRRRRRRRRRRRRRRRRRRRRRRRRRRRRRRRRRRRRRRRRRRRRRRRRRRRRRR
\section{Estimates of Power Exponential Sums}
The estimate for the power sum with relatively prime index is based on the identity
\begin{equation}\label{eq9933Q.210c}
\frac{1}{p} \sum_{0 \leq t\leq p-1,}  \sum_{0 \leq s\leq p-1} \omega^{t(n-s)}f(s)=f(n),\end{equation}
where $\omega=e^{i2\pi/p}$ and $x  \leq p -1$.

%RRRRRRRRRRRRRRRRRRRRRRRRRRRRRRRRRRRRRRRRRRRRRRRRRRRRRRRRRRRRRRRRRRRRR
\subsection{Power Exponential Sum with Relatively Prime Index}
\begin{theorem}  \label{thm9933Q.346}\hypertarget{thm9933Q.346}  Let \(p\geq 2\) be a large prime, and let $\tau $ be a primitive root modulo $p$. If $ x\leq p-1$, then,
	\begin{equation}
		\sum_{\substack{1\leq n\leq x\\\gcd(n,p-1)=1}} e^{i2\pi b \tau^n/p} \ll  p^{1/2+\delta}(\log p)^2 \nonumber,
		\end{equation} 
		where $\delta>0$ is a small real number and the implied constant is independent of $b\ne0$. 	
\end{theorem}
\begin{proof}[\textbf{Proof}] Let $p$ be a large prime, and let $f(n)=e^{i 2 \pi b\tau^{n} /p}$, where $\tau$ is a primitive root modulo $p$. Assume $x=p-1$. Start with the representation
	\begin{equation} \label{eq9933Q.346b}
		\sum_{\substack{1\leq n\leq p-1\\\gcd(n,p-1)=1}} e^{\frac{i2\pi b \tau^n}{p}}= \sum_{\substack{1\leq n\leq p-1\\\gcd(n,p-1)=1}}\frac{1}{p} \sum_{0 \leq t\leq p-1,}  \sum_{1 \leq s\leq p-1} \omega^{t(n-s)}e^{\frac{i2\pi b \tau^s}{p}} ,
	\end{equation}
see \eqref{eq9933Q.210c}. Use the inclusion exclusion principle to rewrite the exponential sum as
	\begin{equation}\label{eq9933Q.346d}
		\sum_{\substack{1\leq n\leq p-1\\\gcd(n,p-1)=1}} e^{ \frac{i2\pi b \tau^n}{p}} 
		= \sum_{1\leq  n \leq p-1}\frac{1}{p} \sum_{0 \leq t\leq p-1,}  \sum_{1 \leq s\leq p-1} \omega^{t(n-s)}e^{\frac{i2\pi b \tau^s}{p}} \sum_{\substack{d \mid p-1 \\ d \mid n}}\mu(d)   .
	\end{equation} 
	Now, observe that the term $t=0$ contributes $-\varphi(p-1)/p$, and rearranging it yield
	\begin{eqnarray}\label{eq9933Q.346f}
		&& \sum_{\substack{1\leq n\leq p-1\\\gcd(n,p-1)=1}} e^{ \frac{i2\pi b \tau^n}{p}} \\
		&=& \sum_{ n \leq p-1}\frac{1}{p} \sum_{1 \leq t\leq p-1,}  \sum_{1 \leq s\leq p-1} \omega^{t(n-s)}e^{\frac{i2\pi b \tau^s}{p}} \sum_{\substack{d \mid p-1 \\ d \mid n}}\mu(d) -\frac{\varphi(p-1)}{p} \nonumber \\
		&=&\frac{1}{p} \sum_{1 \leq t\leq p-1} \left ( \sum_{1 \leq s\leq p-1} \omega^{-ts}e^{\frac{i2\pi b \tau^s}{p}}\right )\left (\sum_{d \mid p-1} \mu(d) \sum_{\substack{n \leq p-1, \\ d \mid n}}   \omega^{tn} \right ) -\frac{\varphi(p-1)}{p} \nonumber.
	\end{eqnarray} 
	Taking absolute value, and applying \hyperlink{lem5555.400B}{Lemma} \ref{lem5555.400B}, and \hyperlink{lem1234A.150A}{Lemma} \ref{lem1234A.150A}, yield
	\begin{eqnarray} \label{eq9933Q.346h}
		&& \left | \sum_{\substack{1\leq n\leq p-1\\\gcd(n,p-1)=1}} e^{\frac{i2\pi b \tau^n}{p}} \right | \\
		&\leq&\frac{1}{p}  \sum_{1 \leq t\leq p-1} \left | \sum_{1 \leq s\leq p-1} \omega^{-ts}e^{i2\pi b \tau^{s}/p} \right | \cdot  \left |\sum_{d \mid p-1} \mu(d) \sum_{\substack{n \leq p-1, \\ d \mid n}}   \omega^{tn} \right | +\frac{\varphi(p-1)}{p}\nonumber \\
		&\ll&\frac{1}{p}  \sum_{1 \leq t\leq p-1} \left ( 2p^{1/2} \log p \right ) \cdot  \left |\sum_{d \mid p-1} \mu(d) \sum_{\substack{n \leq p-1, \\ d \mid n}}   \omega^{tn} \right |+\frac{\varphi(p-1)}{p}\nonumber\\
		&\ll&\frac{1}{p} \left ( 2p^{1/2} \log p \right ) \cdot   \sum_{1 \leq t\leq p-1} \left |\sum_{d \mid p-1} \mu(d) \sum_{\substack{n \leq p-1, \\ d \mid n}}   \omega^{tn} \right |+\frac{\varphi(p-1)}{p}\nonumber\\[.2cm]
		&\ll&\frac{1}{p} \left ( 2p^{1/2} \log p \right ) \cdot  \left ( 2p^{1+\delta} \log p \right ) \nonumber\\[.3cm]
		&\ll& p^{1/2+\delta} (\log p)^2 \nonumber,
	\end{eqnarray}
	where $\delta>0$ is a small number.
\end{proof}

A different approach to this result appears in {\color{red}\cite[Theorem 6]{FS2000}}, and related results are given in \cite{FS2001}, \cite{GM2005}, \cite{CC2009}, and {\color{red}\cite[Theorem 1]{GK2005}}. The upper bound given in \hyperlink{thm9933Q.346}{Theorem} \ref{thm9933Q.346} seems to be optimum. A different proof, which has a weaker upper bound, appears in {\color{red}\cite[Theorem 6]{FS2000}}, and related results are given in \cite{CC2009}, \cite{FS2001}, \cite{GK2005}, and {\color{red}\cite[Theorem 1]{GK2005}}.

\subsection{FFT of Power Exponential Sum with Relatively Prime Index} 
For any fixed $ 0 \ne b \in \mathbb{F}_p$, the map $ \tau^n \longrightarrow b \tau^n$ is one-to-one (permutation) in $\mathbb{F}_p$. Consequently, the subsets 
\begin{equation} \label{eq9933RPI.500b}
	\{ \tau^n: \gcd(n,p-1)=1 \}\quad \text { and } \quad  \{ b\tau^n: \gcd(n,p-1)=1 \} \subset \mathbb{F}_p
\end{equation} have the same cardinalities. As a direct consequence the exponential sums 
\begin{equation} \label{eq9933RPI.500d}
	\sum_{\substack{1\leq n\leq p-1\\\gcd(n,p-1)=1}}e^{i2\pi b \tau^n/p} \quad \text{ and } \quad \sum_{\substack{1\leq n\leq p-1\\\gcd(n,p-1)=1}} e^{i2\pi \tau^n/p},
\end{equation}
have the same upper bound up to an error term. An asymptotic relation for the finite Fourier transform (FFT) of the exponential sums (\ref{eq9933RPI.500d}) is provided here. 

\begin{theorem}   \label{thm9933ERP.220V}\hypertarget{thm9933ERP.220V}  Let \(p\geq 2\) be a large prime. If $\tau $ be a primitive root modulo $p$ and $a<x=o(p)$ is not a primitive root, then
	\begin{equation} 
	\widehat{V(a)}=	\sum_{1\leq b\leq  p-1}	 e^{-i2\pi \frac{b}{p}}	\sum_{\substack{1\leq n\leq p-1\\\gcd(n,p-1)=1}} e^{\frac{i2\pi ab \tau^n}{p}} =-  \sum_{\substack{1\leq n\leq p-1\\\gcd(n,p-1)=1}} e^{\frac{i2\pi  \tau^n}{p}} + O(p^{1/2+\delta} (\log p)^2)\nonumber,
	\end{equation} 
	where $\delta>0$ is a small number and the implied constant is independent of $ b \in [1, p-1]$. 	
\end{theorem} 
\begin{proof}[\textbf{Proof}] For $a\in[1,x]$ and $b\in[1,p-1]$, the exponential sum has the representation 
	\begin{eqnarray} \label{eq9933RPI.500f}
		V(b)&=& \sum_{\substack{1\leq n\leq p-1\\\gcd(n,p-1)=1}} e^{\frac{i2\pi b \tau^n}{p}} \\
		&=&\frac{1}{p} \sum_{1 \leq t\leq p-1} \left ( \sum_{1 \leq s\leq p-1} \omega^{-ts}e^{\frac{i2\pi b \tau^s}{p}}\right )\left (\sum_{d \mid p-1} \mu(d) \sum_{\substack{n \leq p-1, \\ d \mid n}}   \omega^{tn} \right ) -\frac{\varphi(p-1)}{p}\nonumber,
	\end{eqnarray} 
	confer equations \eqref{eq9933Q.346b}, \eqref{eq9933Q.346d} and \eqref{eq9933Q.346f} for more details. In particular, for $b=1$, 
	\begin{eqnarray} \label{eq9933RPI.500h}
		V(1)&=& 	\sum_{\substack{1\leq n\leq p-1\\\gcd(n,p-1)=1}}e^{\frac{i2\pi  \tau^n}{p}} \\
		&=& \frac{1}{p} \sum_{1 \leq t\leq p-1} \left ( \sum_{1 \leq s\leq p-1} \omega^{-ts}e^{\frac{i2\pi  \tau^s}{p}}\right )\left (\sum_{d \mid p-1} \mu(d) \sum_{\substack{n \leq p-1, \\ d \mid n}}   \omega^{tn} \right ) -\frac{\varphi(p-1)}{p}\nonumber,
	\end{eqnarray}
	respectively. Differencing (\ref{eq9933RPI.500f}) and (\ref{eq9933RPI.500h}) produces 
	\begin{eqnarray} \label{eq9933RPI.500i}
		V(b)-V(1)&= &	\sum_{\substack{1\leq n\leq p-1\\\gcd(n,p-1)=1}} e^{\frac{i2\pi b \tau^n}{p}} -\sum_{\substack{1\leq n\leq p-1\\\gcd(n,p-1)=1}} e^{\frac{i2\pi  \tau^n}{p}} \\
		&=&     \frac{1}{p} \sum_{1 \leq t\leq p-1} \left ( \sum_{1 \leq s\leq p-1} \omega^{-ts}e^{\frac{i2\pi  b \tau^s}{p}}-\sum_{1 \leq s\leq p-1} \omega^{-ts}e^{\frac{i2\pi  \tau^s}{p}}\right ) \nonumber \\
		&& \times \left (\sum_{d \mid p-1} \mu(d) \sum_{\substack{n \leq p-1, \\ d \mid n}}   \omega^{tn} \right ) \nonumber.
	\end{eqnarray}
	Taking the finite Fourier transform of the difference $D(b)=V(b)-V(1)$ returns 

	\begin{eqnarray} \label{eq9933RPI.500j}
		\widehat{D(a)}&=&	\sum_{1\leq b\leq  p-1}	 e^{-i2\pi \frac{ab}{p}}\left( \sum_{\substack{1\leq n\leq  p-1\\\gcd(n,p-1)=1}} e^{\frac{i2\pi b \tau^n}{p}} -\sum_{\substack{1\leq n\leq  p-1\\\gcd(n,p-1)=1}} e^{\frac{i2\pi  \tau^n}{p}}\right)  \\
		&=&  \frac{1}{p} \sum_{1\leq b\leq  p-1}	 e^{ \frac{-i2\pi ab}{p}}  \sum_{1 \leq t\leq p-1} \left ( \sum_{1 \leq s\leq p-1} \omega^{-ts}e^{\frac{i2\pi  b \tau^s}{p}}-\sum_{1 \leq s\leq p-1} \omega^{-ts}e^{\frac{i2\pi  \tau^s}{p}}\right ) \nonumber \\
		&&\hskip 1.75in \times \left (\sum_{d \mid p-1} \mu(d) \sum_{\substack{n \leq p-1, \\ d \mid n}}   \omega^{tn} \right ) \nonumber\\
		&=&   \frac{1}{p} \sum_{1 \leq t\leq p-1} \left ( \sum_{1 \leq s\leq p-1} \omega^{-ts}\sum_{1\leq b\leq  p-1}	  e^{\frac{i2\pi b (\tau^s-a)}{p}}\right.  \nonumber \\
		&&\hskip .15in-\left .\sum_{1\leq b\leq  p-1}	 e^{-i2\pi \frac{b}{p}}   \sum_{1 \leq s\leq p-1} \omega^{-ts}e^{\frac{i2\pi  \tau^s}{p}}\right ) \times \left (\sum_{d \mid p-1} \mu(d) \sum_{\substack{n \leq p-1, \\ d \mid n}}   \omega^{tn} \right ) \nonumber.
	\end{eqnarray}
	Now $\tau^s-a\ne0$ for any $s\in\{s<p-1:\gcd(s,p-1)=1\}$ and $a\leq x$. Thus, using the geometric sum identity $\sum_{1\leq u\leq  p-1}	 e^{i2\pi au/p}=-1$ to simplify the last expression yields
	\begin{eqnarray} \label{eq9933RPI.500l}
		\widehat{D(a)}&=&	\sum_{1\leq b\leq  p-1}	 e^{ \frac{-i2\pi ab}{p}}\left( \sum_{\substack{1\leq n\leq  p-1\\\gcd(n,p-1)=1}} e^{\frac{i2\pi b \tau^n}{p}} -\sum_{\substack{1\leq n\leq  p-1\\\gcd(n,p-1)=1}} e^{\frac{i2\pi  \tau^n}{p}}\right)  \nonumber\\
		&=&   \frac{1}{p} \sum_{1 \leq t\leq p-1} \left ( (-1)(-1)-(-1)  \sum_{1 \leq s\leq p-1} \omega^{-ts}e^{\frac{i2\pi  \tau^s}{p}}\right ) \nonumber \\
		&& \times \left (\sum_{d \mid p-1} \mu(d) \sum_{\substack{n \leq p-1, \\ d \mid n}}   \omega^{tn} \right ) .
	\end{eqnarray}
	Rearranging the last equation yield
	\begin{eqnarray} \label{eq9933RPI.500k}
		\widehat{V(a)}&=&\sum_{1\leq b\leq  p-1}	 e^{-i2\pi \frac{ab}{p}}\sum_{\substack{1\leq n\leq  p-1\\\gcd(n,p-1)=1}} e^{\frac{i2\pi  b \tau^n}{p}} \\
		&=& -\sum_{\substack{1\leq n\leq  p-1\\\gcd(n,p-1)=1}} e^{\frac{i2\pi  \tau^n}{p}}  +  \frac{1}{p} \sum_{1 \leq t\leq p-1} \left ( 1-  \sum_{1 \leq s\leq p-1} \omega^{-ts}e^{\frac{i2\pi  \tau^s}{p}}\right ) \nonumber \\
		&&\hskip 2.5in \times \left (\sum_{d \mid p-1} \mu(d) \sum_{\substack{n \leq p-1, \\ d \mid n}}   \omega^{tn} \right ) \nonumber.
	\end{eqnarray}
By \hyperlink{lem5555.400B}{Lemma} \ref{lem5555.400B}, the relatively prime summation kernel is bounded by
	\begin{eqnarray} \label{eq9933RPI.500m}
		\sum_{1 \leq t\leq p-1}	\Bigg |\sum_{d \mid p-1} \mu(d) \sum_{\substack{n \leq p-1, \\ d \mid n}}   \omega^{tn} \Bigg | 
		&=& \sum_{1 \leq t\leq p-1}\Bigg | \sum_{\gcd(n, p-1)=1}\omega^{tn} \Bigg |  \\ 
		&\ll &  p^{1+\delta}\log p\nonumber, 
	\end{eqnarray}
	where $\delta>0$ is a small number and by \hyperlink{lem1234A.150A}{Lemma} \ref{lem1234A.150A}, the difference including Gauss sum is bounded by
	\begin{eqnarray} \label{eq9933RPI.500o}
		\Bigg | 1-  \sum_{1 \leq s\leq p-1} \omega^{-ts}e^{\frac{i2\pi a \tau^s}{p}}\Bigg |=	\Bigg | 1- \sum_{1 \leq s\leq p-1} \chi(s) \psi(s) \Bigg| 
		&\leq & 4 p^{1/2} \log p, 
	\end{eqnarray}
	where  $\chi(s)=e^{i \pi s t/p}$, and $ \psi(s)=e^{i2\pi a \tau^s/p}$. Taking absolute value of the remainder term
	in (\ref{eq9933RPI.500k}) and replacing (\ref{eq9933RPI.500m}), and  (\ref{eq9933RPI.500o}), return
	\begin{eqnarray} \label{eq9933RPI.500p}
		|\widehat{R(a)}|	&=&\frac{1}{p} \Bigg |\sum_{1 \leq t\leq p-1} \Bigg ( 1-  \sum_{1 \leq s\leq p-1} \omega^{-ts}e^{\frac{i2\pi  a\tau^s}{p}}\Bigg ) \cdot\Bigg (\sum_{d \mid p-1} \mu(d) \sum_{\substack{n \leq p-1, \\ d \mid n}}   \omega^{tn} \Bigg ) \Bigg | \nonumber\\
		&=&\frac{1}{p}\sum_{1 \leq t\leq p-1}\Bigg | 1-  \sum_{1 \leq s\leq p-1} \omega^{-ts}e^{\frac{i2\pi a \tau^s}{p}} \Bigg | \cdot \Bigg | \sum_{d \mid p-1} \mu(d) \sum_{\substack{n \leq p-1, \\ d \mid n}}   \omega^{tn} \Bigg | \nonumber\\
		&\ll &\frac{1}{p}(2 p^{1/2} \log p)\cdot \sum_{1 \leq t\leq p-1}\Bigg | \sum_{d \mid p-1} \mu(d) \sum_{\substack{n \leq p-1, \\ d \mid n}}   \omega^{tn} \Bigg | \nonumber\\
		&\ll &\frac{1}{p}(4 p^{1/2} \log p)\cdot (p^{1+\delta} \log p) \nonumber\\[.3cm]
		&\ll & p^{1/2+\delta} (\log p)^2,
	\end{eqnarray}
	where the implied constant depends on the number of divisors of $p-1$.
\end{proof}

%SSSSSSSSSSSSSSSSSSSSSSSSSSSSSSSSSSSSSSSSSSSSSSSSSSSSSSSSSSSSSSSS
%SSSSSSSSSSSSSSSSSSSSSSSSSSSSSSSSSSSSSSSSSSSSSSSSSSSSSSSSSSSSSSSS
%SSSSSSSSSSSSSSSSSSSSSSSSSSSSSSSSSSSSSSSSSSSSSSSSSSSSSSSSSSSSSSSS
%SSSSSSSSSSSSSSSSSSSSSSSSSSSSSSSSSSSSSSSSSSSSSSSSSSSSSSSSSSSSSSSS
%SSSSSSSSSSSSSSSSSSSSSSSSSSSSSSSSSSSSSSSSSSSSSSSSSSSSSSSSSSSSSSSS
%SSSSSSSSSSSSSSSSSSSSSSSSSSSSSSSSSSSSSSSSSSSSSSSSSSSSSSSSSSSSSSSS
\section{Fibers and Multiplicities for Primitive Roots in Finite Fields} \label{S9955PPF}\hypertarget{S9955PPF}
The estimates of the error terms of several results concerning small $k$th power residues in arithmetic progressions and small primitive roots in arithmetic progressions in finite rings and finite fields depend on the multiplicities of the fibers of certain permutation maps. The effective estimates of the multiplicities of these fibers are computed in this section.\\

\begin{lemma} \label{lem9955PPF.300S}\hypertarget{lem9955PPF.300S}  Let $p$ be an odd prime, let $ x=o(p)$ and let $\tau\in \F_p$ be a primitive root in the finite field $\F_p$.  Define the maps
	\begin{equation}\label{eq9955P.300-m}
		\alpha(n,u)\equiv (\tau ^n-u)\bmod p\quad \text{ and } \quad 
		\beta(r,t)\equiv rt\bmod p.
	\end{equation}	
	Then, the fibers $\alpha^{-1}(m)$ and $\beta^{-1}(m)$ of an element $0\ne m\in \F_p$
	have the cardinalities 
	\begin{equation}\label{eq9955P.300-f}
		\#	\alpha^{-1}(m)\leq x-1\quad \text{ and }\quad \#\beta^{-1}(m)=	x
	\end{equation}
	respectively.
\end{lemma}
\begin{proof}[\textbf{Proof}] Let $\mathscr{R}=\{n<p:\gcd(n,p-1)=1\}$. Given a fixed $u\in [2,x]$, the map 
	\begin{equation}\label{eq9955P.300-m1}
		\alpha:\mathscr{R}\times [2,x] \longrightarrow\F_p\quad  \text{ defined by }\quad  \alpha(n,u)\equiv (\tau ^n-u)\bmod p,
	\end{equation}
	is one-to-one. This follows from the fact that the map $n\longrightarrow\tau^n \bmod p$ is a permutation of the nonezero elements of the finite field $\F_p$, and the restriction map $n\longrightarrow(\tau ^n-u)\bmod p$ is a shifted permutation, it maps the subset \begin{equation}\label{eq9955P.300-p}
		\mathscr{R}\subset \F_p\quad \text{ to }\quad \mathscr{R}-u\subset \F_p,
	\end{equation} see {\color{red}\cite[Chapter 7]{LN1997}} for extensive details on the theory of permutation functions of finite fields. Thus, as $(n,u)\in \mathscr{R}\times [2,x]$ varies, a value $m=\alpha(n,u)\in \F_p$ is repeated at most $x-1$ times. Moreover, the premises no primitive root $u\leq x$ implies that $m=\alpha(n,u)\ne0$. This verifies that the cardinality of the fiber is 
	\begin{eqnarray}\label{eq9955P.300-f1}
		\#	\alpha^{-1}(m)&=&	\#\{(n,u):m\equiv (\tau ^n-u)\bmod p:2\leq u\leq x \text{ and }\gcd(n,p-1)=1\}\nonumber\\[.3cm]&\leq& x-1.
	\end{eqnarray}		
	Similarly, given a fixed $a\in [1,x]$, the map 
	\begin{equation}\label{eq9955P.300-m2}
		\beta:[1,x]\times [1,p-1]\longrightarrow\F_p\quad  \text{ defined by }\quad  \beta(r,t)\equiv rt\bmod p,
	\end{equation}
	is one-to-one. Here the map $t\longrightarrow rt \bmod p$ permutes the elements of the finite field $\F_p$. Thus, as $(r,t)\in [1,x]\times [1,p-1]$ varies, each value $m=\beta(r,t)\in \F_p^{\times}$ is repeated exactly $x$ times. This verifies that the cardinality of the fiber is 
	\begin{equation}\label{eq9955P.300-f2}
		\#	\beta^{-1}(m)=	\#\{(r,t):m\equiv rt\bmod p:1\leq r\leq x \text{ and }1\leq t< p\}=x
	\end{equation}
	
	Now each value $m=\alpha(n,u)\ne0$ (of multiplicity up to $(x-1)$ in $	\alpha^{-1}(m)$), is matched to $m=\alpha(n,u)=\beta(r,t)$ for some $(r,t)$, 
	(of multiplicity exactly $x$ in $	\beta^{-1}(m)$). Now, comparing \eqref{eq9955P.300-f1} and \eqref{eq9955P.300-f2} proves that $\# \alpha^{-1}(m)\leq\# \beta^{-1}(m)$. 
\end{proof}

%SSSSSSSSSSSSSSSSSSSSSSSSSSSSSSSSSSSSSSSSSSSSSSSSSSSSSSSSSSSSSSSS
%SSSSSSSSSSSSSSSSSSSSSSSSSSSSSSSSSSSSSSSSSSSSSSSSSSSSSSSSSSSSSSSS
%SSSSSSSSSSSSSSSSSSSSSSSSSSSSSSSSSSSSSSSSSSSSSSSSSSSSSSSSSSSSSSSS
%SSSSSSSSSSSSSSSSSSSSSSSSSSSSSSSSSSSSSSSSSSSSSSSSSSSSSSSSSSSSSSSS
%SSSSSSSSSSSSSSSSSSSSSSSSSSSSSSSSSSSSSSSSSSSSSSSSSSSSSSSSSSSSSSSS
%SSSSSSSSSSSSSSSSSSSSSSSSSSSSSSSSSSSSSSSSSSSSSSSSSSSSSSSSSSSSSSSS
\section{Lower Bound of the Totient Function}\label{S9955PT}\hypertarget{S9955PT}
This section provides a detailed proof of the lower bound of the totient function $\varphi(n)/n=\prod_{r\mid n}\left( 1-1/r\right)$, where $r\geq2$ is a prime divisor of $n$, see {\color{red}\cite[Theorem 2.4]{AT1976}}.  
%TTTTTTTTTTTTTTTTTTTTTTTTTTTTTTTTTTTTTTTTTTTTTTTTTTTTTTTT
\begin{lemma}  \label{lem9955P.400T}\hypertarget{lem9955P.400T} If $p$ is a large prime, then
	$$\frac{\varphi(p-1)}{p}\gg\frac{1}{\log \log p}.$$
\end{lemma}
\begin{proof}[\textbf{Proof}] For any prime $p$, the ratio $\varphi(p-1)/p$ can be rewritten as a product over the prime
	\begin{equation}\label{eq9955P.400Tc}
		\frac{\varphi(p-1)}{p}=		\frac{p-1}{p}\cdot \frac{\varphi(p-1)}{p-1}=\frac{p-1}{p}\prod_{r\mid p-1}\left( 1-\frac{1}{r}\right).
	\end{equation}	
	where $r\geq2$ ranges over the prime divisor of $p-1$. This step follows from the identity $\varphi(n)/n=\prod_{r\mid n}\left( 1-1/r\right)$, where $r\geq2$ ranges 
	over the prime divisors of $n$. Since the number $p-1$ has fewer than $ 2\log p$ prime divisors, see {\color{red}\cite[Theorem 2.10]{MV2007}}, let $ x=2\log p$. Then, an application of the lower bound of the product given in {\color{red}\cite[Theorem 6.12]{DP2016}} yields 
	\begin{eqnarray}\label{eq9955P.400Tf}
		\frac{\varphi(p-1)}{p}&\geq&\frac{p-1}{p}\prod_{r\leq 2\log p}\left( 1-\frac{1}{r}\right)\\[.3cm]
		&>&\frac{p-1}{p}\cdot \frac{e^{-\gamma}}{\log( 2\log p)}\left(1-\frac{0.2}{(\log (2\log p))^2} \right)\nonumber\\[.3cm]
		& \gg&\frac{1}{\log \log p}>0 \nonumber,
	\end{eqnarray}	
	where $\gamma>0$ is Euler constant.	
\end{proof}
An alternative result for the lower bound of the ratio $\varphi(n)/n$ appears in {\color{red}\cite[Theorem 2.9]{MV2007}}.

\begin{rmk}{\normalfont The explicit lower bound $p\geq p_0>2^{2145}\approx 10^{645}$ for the parameter $ p_0>0 $ is derived from the explicit estimate $ x=2\log p_0>2973 $ for the totient product given in 
		{\color{red}\cite[Theorem 6.12]{DP2016}} and the extreme value of the prime divisors counting function $\omega(p-1)\leq 2\log p$. In practice the prime numbers $p$ are significantly smaller than the explicit lower bound. In fact, the numerical data in {\color{red}\cite[Table 1]{MS2022B}} suggests that the magnitude of the least primitive root modulo $p$ stated in \hyperlink{thm9955P.800}{Theorem} \ref{thm9955P.800} is nearly independent of the ratio $\varphi(p-1)/p$. In addition, the average order of the latter is much smaller, that is, $ x=(\log \log p)^3/2$, see {\color{red}\cite[Theorem 3.1]{EP1985}} for more details, which implies that on average, ratio $\varphi(p-1)/p$ remains nearly constant as $p$ varies over the set of primes $\tP=\{2,3,5,7,\ldots \}$.
	}
\end{rmk}
%SSSSSSSSSSSSSSSSSSSSSSSSSSSSSSSSSSSSSSSSSSSSSSSSSSSSSSSSSSSSSSSS
%SSSSSSSSSSSSSSSSSSSSSSSSSSSSSSSSSSSSSSSSSSSSSSSSSSSSSSSSSSSSSSSS
%SSSSSSSSSSSSSSSSSSSSSSSSSSSSSSSSSSSSSSSSSSSSSSSSSSSSSSSSSSSSSSSS
%SSSSSSSSSSSSSSSSSSSSSSSSSSSSSSSSSSSSSSSSSSSSSSSSSSSSSSSSSSSSSSSS
%SSSSSSSSSSSSSSSSSSSSSSSSSSSSSSSSSSSSSSSSSSSSSSSSSSSSSSSSSSSSSSSS
%SSSSSSSSSSSSSSSSSSSSSSSSSSSSSSSSSSSSSSSSSSSSSSSSSSSSSSSSSSSSSSSS
\section{Evaluation of the Main Term}
%TTTTTTTTTTTTTTTTTTTTTTTTTTTTTTTTTTTTTTTTTTTTTTTTTTTTTTTTTTTTTTTTTTTT
\begin{lemma} \label{lem9955P.300T}\hypertarget{lem9955P.300T} Let  \(p\geq 2\) be a large prime, let \(u\leq x=(\log p)(\log\log p)^5\) be a small real number and let $q\ll\log\log p$. If $1 \leq a <q$ be a pair of relatively prime integers, then 
	\begin{equation}
		\sum_{\substack{2 \leq u\leq x\\
				u \equiv a \bmod q		}} \frac{1}{p}\sum_{\substack{1\leq n\leq p-1\\\gcd(n,p-1)=1}}\Lambda(u)
		= \frac{\varphi(p-1)}{p}  \left( \frac{ x}{\varphi(q)}+O\left( xe^{-c\sqrt{\log x}}\right)\right)  ,\nonumber
	\end{equation}
	where $c>0$ is a constant.
\end{lemma}

\begin{proof}[\textbf{Proof}] The number of relatively integers $n<p$ coincides with the values of the totient function. A routine rearrangement gives 
	\begin{eqnarray}\label{eq9955N.300f}
		\sum_{\substack{2 \leq u\leq x\\
				u \equiv a \bmod q		}} \frac{1}{p}\sum_{\substack{1\leq n\leq p-1\\\gcd(n,p-1)=1}}\Lambda(u) 
		&=&\frac{\varphi(p-1)}{p}\sum_{\substack{2 \leq u\leq x\\
				u \equiv a \bmod q		}} \Lambda(u)\\[.3cm]
		&=& \frac{\varphi(p-1)}{p}  \cdot \left(   \frac{x}{\varphi(q)}+O\left( xe^{-c\sqrt{\log x}}\right)\right) \nonumber.
	\end{eqnarray}
	The last line follows from  {\color{red}\cite[Corollary 5.29.]{IK2004}}, {\color{red}\cite[Corollary 11.19]{MV2007}}, et cetera. 
\end{proof}

%SSSSSSSSSSSSSSSSSSSSSSSSSSSSSSSSSSSSSSSSSSSSSSSSSSSSSSSSSSSSSSSS
%SSSSSSSSSSSSSSSSSSSSSSSSSSSSSSSSSSSSSSSSSSSSSSSSSSSSSSSSSSSSSSSS
%SSSSSSSSSSSSSSSSSSSSSSSSSSSSSSSSSSSSSSSSSSSSSSSSSSSSSSSSSSSSSSSS
%SSSSSSSSSSSSSSSSSSSSSSSSSSSSSSSSSSSSSSSSSSSSSSSSSSSSSSSSSSSSSSSS
%SSSSSSSSSSSSSSSSSSSSSSSSSSSSSSSSSSSSSSSSSSSSSSSSSSSSSSSSSSSSSSSS
\section{Estimate of the Error Term} \label{S9955PAP-ET}\hypertarget{S9955PAP-ET}
A nontrivial upper bound of the error term is computed in this section. To achieve that the error term is partitioned as $E(x,q,a)=E_{0}(x,q,a)+E_{1}(x,q,a)$. The upper bound of the first term $E_0(x,q,a)$ for $n<p/x$ is derived using geometric summation/sine approximation techniques, and the upper bound of the second term $E_1(x,q,a)$ for $p/x\leq n\leq p$ is derived using exponential sums techniques.

\begin{lemma}  \label{lem9955PAP.300E}\hypertarget{lem9955PAP.300E} Let \(p\geq 2\) be a large prime,  let $x= (\log p)(\log\log p)^{4+\varepsilon}$ be a real number and let $q\leq \log \log p$. If $1\leq a<q$ is a pair relatively prime integers and there is no primitive root \(u\leq x\) then 
	\begin{equation}\label{eq9955PAP.300b}
		\sum_{\substack{2 \leq u\leq x \\ u \equiv a \bmod q}}
		\frac{\Lambda(u)}{p}\sum_{\substack{1\leq s\leq p-1\\\gcd(s,p-1)=1}} \sum_{ 1\leq t\leq p-1} \psi \left((\tau^s-u)t\right) \ll (\log p)(\log x)^2\nonumber, 
	\end{equation} 
	where $\psi(s)=e^{i 2 \pi ks/p}$ with $0< k<p$, is an additive character.
\end{lemma}

\begin{proof}[\textbf{Proof}] The product of a point $(a,b)\in [1,x]\times [1,p/x)$ satisfies $ab<p$. This leads to the partition $[1,p/x)\cup[p/x,p)$ of the index $s$, which is suitable for the sine approximation $ab/p\ll\sin(\pi ab/p)\ll ab/p$ for $|ab/p|<1$ on the first subinterval $[1,p/x)$, see \eqref{eq9955PAP.700u1}. Thus, consider the partition of the triple finite sum
	\begin{eqnarray} \label{eq9955PAP.300k}
		E(x,q,a)&=& \sum_{\substack{2 \leq u\leq x \\ u \equiv a \bmod q}}
		\frac{\Lambda(u)}{p}\sum_{\substack{1\leq s\leq p-1\\\gcd(s,p-1)=1}} \sum_{ 1\leq t\leq p-1} e^{i2\pi \frac{(\tau^s-u)t}{p}}   \\[.3cm]
		&= & \sum_{\substack{2 \leq u\leq x \\ u \equiv a \bmod q}}
		\frac{\Lambda(u)}{p}\sum_{\substack{1\leq s< p/x\\\gcd(s,p-1)=1}} \sum_{ 1\leq t\leq p-1} e^{i2\pi \frac{(\tau^s-u)t}{p}} \nonumber\\[.3cm]
		&&\hskip 1.15 in+ \sum_{\substack{2 \leq u\leq x \\ u \equiv a \bmod q}}
		\frac{\Lambda(u)}{p}\sum_{\substack{p/x\leq s\leq p-1\\\gcd(s,p-1)=1}} \sum_{ 1\leq t\leq p-1} e^{i2\pi \frac{(\tau^s-u)t}{p}} \nonumber\\[.3cm]
		&=&E_{0}(x,q,a)\;+\;E_{1}(x,q,a) \nonumber.
	\end{eqnarray} 
	The first suberror term $E_0(x,q,a)$ is estimated in \hyperlink{lem9955PAP.700}{Lemma} \ref{lem9955PAP.700} and the second suberror term $E_1(x,q,a)$ is estimated in \hyperlink{lem9955PAP.750}{Lemma} \ref{lem9955PAP.750}.  Summing these estimates yields
	\begin{eqnarray} \label{eq9955PAP.300u4}
		E(x,q,a)&=& E_{0}(x,q,a)\;+\;E_{1}(x,q,a,)   \\[.2cm]
		&\ll&  (\log x)^2(\log p)\;+\; \frac{(\log p)^2}{p^{1/2-\delta}} (x\log x) \nonumber\\[.12cm]
		&\ll& (\log x)^{2}(\log p)\nonumber,
	\end{eqnarray}
	where $\delta>0$ is a small real number. This completes the estimate of the error term .
\end{proof}

%LLLLLLLLLLLLLLLLLLLLLLLLLLLLLLLLLLLLLLLLLLLLLLLLLLLLLLLLLLLLLLLLLLLLLLLLL
\begin{lemma}   \label{lem9955PAP.700}\hypertarget{lem9955PAP.700}  Let \(p\geq 2\) be a large prime,  let $x< p$ be a real number and let $q\leq \log \log p$. If $1\leq a<q$ is a pair relatively prime integers and there is no primitive root \(u\leq x\) then 
	\begin{equation} 
		E_0(x,q,a)= \sum_{\substack{2 \leq u\leq x \\ u \equiv a \bmod q}}
		\frac{\Lambda(u)}{p}\sum_{\substack{1\leq s< p/x\\\gcd(s,p-1)=1}} \sum_{ 1\leq t\leq p-1} e^{i2\pi \frac{(\tau^s-u)t}{p}}=O\left( (\log x)^2(\log p)\right) \nonumber.
	\end{equation} 
\end{lemma} 
\begin{proof}[\textbf{Proof}] To apply the geometric summation/sine approximation techniques, the subsum $E_0(x,q,a)$ is partition as follows.
	\begin{eqnarray} \label{eq9955PAP.700l}
		E_0(x,q,a)&=& \sum_{\substack{2 \leq u\leq x \\ u \equiv a \bmod q}}
		\frac{\Lambda(u)}{p}\sum_{\substack{1\leq s< p\\\gcd(s,p-1)=1}} \sum_{ 1\leq t\leq p-1} e^{i2\pi \frac{(\tau^s-u)t}{p}}   \\[.3cm]
		&= & \sum_{\substack{2 \leq u\leq x \\ u \equiv a \bmod q}}
		\frac{\Lambda(u)}{p}\sum_{\substack{1\leq s< p\\\gcd(s,p-1)=1}} \left( \sum_{ 1\leq t\leq p/2} e^{i2\pi \frac{(\tau ^s-u)t}{p}}+ \sum_{ p/2<t\leq p-1} e^{i2\pi \frac{(\tau^s-u)t}{p}}\right) \nonumber\\[.3cm]
		&=&E_{0,0}(x,q,a)\;+\;E_{0,1}(x,q,a) \nonumber.
	\end{eqnarray} 
	
	Now, a geometric series summation of the inner finite sum in the first term yields
	\begin{eqnarray} \label{eq9955PAP.700m}
		E_{0,0}(x,q,a)&=& \sum_{\substack{2 \leq u\leq x \\ u \equiv a \bmod q}}
		\frac{\Lambda(u)}{p}\sum_{\substack{1\leq s< p\\\gcd(s,p-1)=1}}  \sum_{ 1\leq t\leq p/2} e^{i2\pi \frac{(\tau^s-u)t}{p}}  \\[.3cm]
		&=&   \sum_{\substack{2 \leq u\leq x \\ u \equiv a \bmod q}}\frac{\Lambda(u)}{p}\sum_{\substack{1\leq s<p\\\gcd(s,p-1)=1}}   \frac{e^{i2\pi (\frac{\tau^s-u}{p})(\frac{p}{2}+1)}-e^{i2\pi \frac{(\tau^s- u)}{p}}}{1-e^{i2\pi \frac{(\tau^s-u)}{p}}} \nonumber\\[.3cm]
		&\leq&   	\frac{\log x}{p} \sum_{\substack{2 \leq u\leq x \\ u \equiv a \bmod q,}}\sum_{\substack{1\leq s< p\\\gcd(s,p-1)=1}}   \frac{2}{|\sin\pi(\tau^s-u)/p|} \nonumber,
	\end{eqnarray} 
	see {\color{red}\cite[Chapter 23]{DH2000}} for similar geometric series calculation and estimation. The last line in \eqref{eq9955PAP.700m} follows from the hypothesis that $u$ is not a primitive root. Specifically, $0\ne s-\log_{\tau} u\in \F_p$ for any $s \geq 1$ such that $\gcd(s,p-1)=1$ and any $u\leq x$. Utilizing \hyperlink{lem9955PPF.300S}{Lemma} \ref{lem9955PPF.300S}, the first term has the upper bound
	\begin{eqnarray} \label{eq9955PAP.700u1}
		E_{0,0}(x,q,c)&\ll&\frac{\Lambda(u)}{p} \sum_{\substack{2 \leq u\leq x \\ u \equiv c \bmod q}}\sum_{\substack{1\leq s<p/x\\\gcd(s,p-1)=1}}   \frac{2}{|\sin\pi(\tau^s-u)/p|}\\	[.3cm]
		&\ll&  	\frac{2\log x}{p} \sum_{1\leq a\leq x}\sum_{1\leq b< p/x}   \frac{1}{|\sin\pi ab/p|}\nonumber\\	[.3cm]
		&\ll&  	\frac{2\log x}{p} \sum_{1\leq a\leq x}\sum_{1\leq b< p}   \frac{p}{\pi ab} \nonumber\\	[.3cm]
		&\ll& (\log x)	\sum_{1\leq a\leq x}\frac{1}{a}\sum_{1\leq b< p}   \frac{1}{b} \nonumber\\	[.3cm]
		&\ll& (\log x)^2(\log p)\nonumber,
	\end{eqnarray}
	where $ab<p$ and $|\sin\pi ab/p|\ne0$ since $p\nmid ab$. Similarly, the second term has the upper bound
	\begin{eqnarray} \label{eq9955PAP.700v}
		E_{0,1}(x,q,c)&=& \sum_{\substack{2 \leq u\leq x \\ u \equiv c \bmod q}}
		\frac{\Lambda(u)}{p}\sum_{\substack{1\leq s<p/x\\\gcd(s,p-1)=1}}  \sum_{ p/2<t\leq p-1} e^{i2\pi \frac{(\tau ^s-u)t}{p}}  \\[.3cm]
		&=&   	\frac{\Lambda(u)}{p} \sum_{\substack{2 \leq u\leq x \\ u \equiv c \bmod q}}\sum_{\substack{1\leq s<p/x\\\gcd(s,p-1)=1}}   \frac{e^{i2\pi \frac{(\tau ^s-u)}{p}}-e^{i2\pi (\frac{\tau ^s-u}{p})(\frac{p}{2}+1)}}{1-e^{i2\pi \frac{(\tau ^s-u)}{p}}} \nonumber\\[.3cm]
		&\leq&   	\frac{\log x}{p} \sum_{\substack{2 \leq u\leq x \\ u \equiv c \bmod q}}\sum_{\substack{1\leq s<p/x\\\gcd(s,p-1)=1}}  \frac{2}{|\sin\pi(\tau ^s-u)/p|} \nonumber\\[.3cm]
		&\ll&  (\log x)^2(\log p)\nonumber.
	\end{eqnarray}
	This is computed in the way as done in \eqref{eq9955PAP.700m} to \eqref{eq9955PAP.700u1}, mutatis mutandis. Thus, \\		
	\begin{equation}
		E_{0}(x,q,a)	=E_{0,0}(x,q,a,)\;+\;E_{0,1}(x,q,a)\ll  (\log x)^2(\log p).
	\end{equation}
This completes the verification.
	\end{proof}
%LLLLLLLLLLLLLLLLLLLLLLLLLLLLLLLLLLLLLLLLLLLLLLLLLLLLLLLLLLLLLLLLLLLLLLLL
\begin{lemma}   \label{lem9955PAP.750}\hypertarget{lem9955PAP.750} Let \(p\geq 2\) be a large prime,  let $x< p$ be a real number and let $q\leq \log \log p$. If $1\leq a<q$ is a pair relatively prime integers and there is no primitive root \(u\leq x\) then 
	\begin{equation} 
		E_{1}(x,q,a)=\sum_{\substack{2 \leq u\leq x \\ u \equiv a \bmod q}}	\frac{\Lambda(u)}{p}\sum_{\substack{p/x\leq s\leq p-1\\\gcd(s,p-1)=1}}  \sum_{ 1\leq t\leq p-1} e^{i2\pi \frac{(\tau^s-u)t}{p}}  = 	 O\left( \frac{(\log p)^2}{p^{1/2-\delta}} (x\log x) \right) \nonumber,
	\end{equation} 
	where $\delta>0$ is a small real number. 	
\end{lemma} 
\begin{proof}[\textbf{Proof}] The hypothesis $\tau ^s-u\ne0$ in the finite field $\F_p$ for any $u\leq x=o(p)$ and $\gcd(n,p-1)=1$, implies that $E_1(x,q,a)$ has a nontrivial upper bound. To determine a nontrivial upper bound, rearrange the triple finite sum $E_{1}(x,q,a)$ and apply the result for the finite Fourier transform in \hyperlink{thm9933ERP.220V}{Theorem} \ref{thm9933ERP.220V}, for example set $a=u$ and take $\widehat{V(u)}=\widehat{V(a)}$, to the new inner sum on the second line below:
	\begin{eqnarray} \label{eq9955PAP.750d}
		E_{1}(x,q,a)&=&\sum_{\substack{2 \leq u\leq x \\ u \equiv a \bmod q}}
		\frac{\Lambda(u)}{p}\sum_{\substack{p/x\leq s\leq p-1\\\gcd(s,p-1)=1}}  \sum_{ 1\leq t\leq p-1} e^{i2\pi \frac{(\tau^s-u)t}{p}}\\	[.3cm]
		&=&   \sum_{\substack{2 \leq u\leq x \\ u \equiv a \bmod q}} \frac{\Lambda(u)}{p}\sum _{1\leq t\leq p-1}e^ {\frac{-i2\pi ut}{p}} \sum_{\substack{p/x\leq s\leq p-1\\\gcd(s,p-1)=1}}  e^ {i2 \pi t\frac{\tau ^{s}}{p}}  \nonumber\\[.3cm]
		&=&   \sum_{\substack{2 \leq u\leq x \\ u \equiv a \bmod q}} \frac{\Lambda(u)}{p}\left(- \sum_{\substack{p/x\leq s\leq p-1\\\gcd(s,p-1)=1}}  e^ {i2 \pi \frac{u\tau ^{s}}{p}} + O\left( p^{1/2+\delta}(\log p)^2\right) \right)  \nonumber ,
	\end{eqnarray}
	where $\delta>0$ is a small real number. Take absolute value and apply the triangle inequality:
	\begin{eqnarray} \label{eq9955PAP.750f}
		|E_{1}(x,q,a)|
		&\leq&   \sum_{\substack{2 \leq u\leq x \\ u \equiv a \bmod q}}\frac{\Lambda(u)}{p}  \left|\sum_{\substack{p/x\leq s\leq p-1\\\gcd(s,p-1)=1}}   e^ {i2 \pi \frac{u \tau ^{s}}{p}} + O\left( p^{1/2+\delta}(\log p)^2\right) \right| \nonumber\\[.3cm]
		&\ll& \frac{\log x}{p} \sum_{\substack{2 \leq u\leq x \\ u \equiv a \bmod q}} \left( \left| \sum_{\substack{p/x\leq s\leq p-1\\\gcd(s,p-1)=1}}   e^ {i2 \pi \frac{u\tau ^{s}}{p}}\right| +  \left|p^{1/2+\delta} (\log p)^2\right|\right)  \nonumber \\[.3cm]
		&\ll&   \frac{\log x}{p} \sum_{\substack{2 \leq u\leq x \\ u \equiv a \bmod q}}
		\left|p^{1/2+\delta}(\log p)^2\right|\\[.3cm]
		&\ll&   \frac{(\log p)^2}{p^{1/2-\delta}} (x\log x) \nonumber,
	\end{eqnarray}
	where the exponential sum estimate
	\begin{eqnarray} \label{eq9955PAP.750j}
\left| \sum_{\substack{p/x\leq s\leq p-1\\\gcd(s,p-1)=1}}   e^ {i2 \pi \frac{u\tau ^{s}}{p}}\right|&\leq & \left| \sum_{\substack{1\leq s\leq p-1\\\gcd(s,p-1)=1}}   e^ {i2 \pi \frac{u\tau ^{s}}{p}}-\sum_{\substack{1\leq s\leq p/x\\\gcd(s,p-1)=1}}   e^ {i2 \pi \frac{u\tau ^{s}}{p}}\right|\nonumber\\[.3cm]&\leq &2 \left| \sum_{\substack{1\leq s\leq p-1\\\gcd(s,p-1)=1}}   e^ {i2 \pi \frac{u\tau ^{s}}{p}}\right|\nonumber\\[.3cm]
&\ll&		\left| \sum_{\substack{1\leq s\leq p-1\\\gcd(s,p-1)=1}}   e^ {i2 \pi \frac{u\tau ^{s}}{p}}\right|\ll p^{1/2+\delta}(\log p)^2
	\end{eqnarray} 
	follows from \hyperlink{thm9933Q.346}{Theorem} \ref{thm9933Q.346}.
\end{proof}

%SSSSSSSSSSSSSSSSSSSSSSSSSSSSSSSSSSSSSSSSSSSSSSSSSSSSSSSSSSSSSSSS
%SSSSSSSSSSSSSSSSSSSSSSSSSSSSSSSSSSSSSSSSSSSSSSSSSSSSSSSSSSSSSSSS
%SSSSSSSSSSSSSSSSSSSSSSSSSSSSSSSSSSSSSSSSSSSSSSSSSSSSSSSSSSSSSSSS
%SSSSSSSSSSSSSSSSSSSSSSSSSSSSSSSSSSSSSSSSSSSSSSSSSSSSSSSSSSSSSSSS
%SSSSSSSSSSSSSSSSSSSSSSSSSSSSSSSSSSSSSSSSSSSSSSSSSSSSSSSSSSSSSSSS
\section{Prime Primitive Roots in Arithmetic Progressions} \label{S9955P}\hypertarget{S9955P}
The determination of an upper bound for the smallest prime primitive root in arithmetic progressions requires the weighted characteristic function of prime numbers, (better known as the vonMagoldt function),
\begin{equation} \label{eq9955P.400i}
	\Lambda(n)=
	\begin{cases}
		\log p&\text{ if } n=p^k,\\
		0&\text{ if } n\ne p^k,
	\end{cases}
\end{equation}
where $p^k$ is a prime power. Let $g^*(p)\geq2$ denotes the smallest prime primitive root $\bmod p$. Define the counting function
\begin{equation} \label{eq9955P.800h}
	N_0(x,q,a)	=\#\{u=qn+a\leq x: \ord_p u=p-1\text{ and }\Lambda(u)\ne0\}.
\end{equation}	
\begin{proof}[\textbf{Proof of \hyperlink{thm9955P.800}{Theorem} \ref{thm9955P.800}}] Let \(p>2\) be a large prime number and let $x=(\log p)(\log\log  p)^{4+\varepsilon}$. Suppose the least primitive root $u>x$ and consider the sum of the characteristic function over the short interval \([2,x]\), that is, 
	\begin{equation} \label{eq9955P.400h}
		N_p(x,q,a)=\sum _{\substack{2 \leq u\leq x\\
				u \equiv a \bmod q		}} \Psi (u)\Lambda(u)=0.
	\end{equation}
	Notice that the prime power $u^a\geq2^a$, with odd $a\geq1$, detected by $ \Psi (u)\Lambda(u)$ and $u\geq2$ are both primitive roots. Replacing the characteristic function, \hyperlink{lem9955FF.200DFL}{Lemma} \ref{lem9955FF.200DFL}, and expanding the nonexistence equation \eqref{eq9955P.400h} yield
	\begin{eqnarray} \label{eq9955P.400m}
		N_p(x,q,a)&=&\sum _{\substack{2 \leq u\leq x\\
				u \equiv a \bmod q		}} \Psi (u) \Lambda(u) \\
		&=&\sum_{\substack{2 \leq u\leq x\\
				u\equiv a \bmod q		}} \left (\frac{1}{p}\sum_{\substack{1\leq s\leq p-1\\\gcd(s,p-1)=1}} \sum_{ 0\leq t\leq p-1} \psi \left((\tau^s-u)\right) \right ) \Lambda(u)\nonumber\\[.3cm] 
		&=&\sum_{\substack{2 \leq u\leq x\\
				p \equiv a \bmod q		}} \frac{1}{p}\sum_{\substack{1\leq s\leq p-1\\\gcd(s,p-1)=1}} \Lambda(u)  \nonumber\\[.3cm]
		&&\hskip .75in +\sum_{\substack{2 \leq u\leq x\nonumber\\[.3cm] 
				u \equiv a \bmod q		}}
		\frac{1}{p}\sum_{\substack{1\leq s\leq p-1\\\gcd(s,p-1)=1}} \sum_{ 1\leq t\leq p-1} \psi \left((\tau^s-u)t\right)\Lambda(u)\nonumber\\[.3cm] 
		&=&M(x,q,a)\; +\; E(x,q,a)\nonumber.
	\end{eqnarray} 
	
	The main term $M(x,q,a)$, which is determined by a finite sum over the trivial additive character \(\psi(t) =1\), is computed in \hyperlink{lem9955P.300T}{Lemma} \ref{lem9955P.300T}, and the error term $E(x,q,a)$, which is determined by a finite sum over the nontrivial additive characters \(\psi(t) =e^{i 2\pi  st/p}\neq 1\), is computed in \hyperlink{lem9955PAP.300E}{Lemma} \ref{lem9955PAP.300E}. \\
	
	Substituting these evaluation and upper bound yield and replacing the real number $x=(\log p)(\log\log p)^{4+\varepsilon}$ yield
	\begin{eqnarray} \label{eq9955P.400p}
		N_p(x,q,a)	
		&=&M(x,q,a) + E(x,q,a) \\[.3cm]		
		&=&		\frac{\varphi(p-1)}{p}  \cdot \left(   \frac{x}{\varphi(q)}+O\left( xe^{-c\sqrt{\log x}}\right)\right)  + O(\log p)(\log x)^2\nonumber\\[.3cm]
		&=&\frac{\varphi(p-1)}{p}  \cdot \frac{(\log p)(\log\log p)^{4+\varepsilon}}{\varphi(q)}  \cdot  \left(   1+O\left( \varphi(q)e^{-c_0\sqrt{\log\log p}}\right)\right)   \nonumber\\[.3cm]
		&& \hskip 2.5 in +O\left( (\log p)(\log \log p)^2 \right) \nonumber.
	\end{eqnarray} 
	Applying \hyperlink{lem9955P.400T}{Lemma} \ref{lem9955P.400T} to the totient function and $\varphi(q)\leq \log\log p$ show that the main term in \eqref{eq9955P.400p} dominates the error term:
	
	\begin{eqnarray} \label{eq9955P.400t}
		N_p(x,q,a)	
		&=&M(x,q,a) + E(x,q,a) \\[.3cm]		
		&\gg&\frac{1}{\log\log p}  \cdot \frac{(\log p)(\log\log p)^{4+\varepsilon}}{\log\log p}  \cdot  \left(   1+O\left( (\log\log p)e^{-c_0\sqrt{\log\log p}}\right)\right)   \nonumber\\[.3cm]
		&& \hskip 2.5 in +O\left( (\log p)(\log \log p)^2 \right)   \nonumber\\[.3cm]
		&\gg&  (\log p)(\log \log p)^{2+\varepsilon} \nonumber\\[.3cm]
		&>&0\nonumber 
	\end{eqnarray} 
	as $p\to\infty$, where $c_0>0$ is an constant. Clearly, this contradicts the hypothesis \eqref{eq9955P.400h} for all sufficiently large prime numbers $p \geq p_0$. Therefore, there exists a small prime primitive root $u=qn+a\leq x\ll(\log p)(\log\log p)^5$.
\end{proof}

The best and closest result in this direction in the literature is the existence of a subset of prime primitive roots described in \eqref{eq9955P.800c} and \eqref{eq9955P.800d}.

%SSSSSSSSSSSSSSSSSSSSSSSSSSSSSSSSSSSSSSSSSSSSSSSSSSSSSSSSSSSSSSSSSSSSSSSSSSSSSSSSSSS
%SSSSSSSSSSSSSSSSSSSSSSSSSSSSSSSSSSSSSSSSSSSSSSSSSSSSSSSSSSSSSSSSSSSSSSSSSSSSSSSSSSS
%SSSSSSSSSSSSSSSSSSSSSSSSSSSSSSSSSSSSSSSSSSSSSSSSSSSSSSSSSSSSSSSSSSSSSSSSSSSSSSSSSSS
%SSSSSSSSSSSSSSSSSSSSSSSSSSSSSSSSSSSSSSSSSSSSSSSSSSSSSSSSSSSSSSSSSSSSSSSSSSSSSSSSSSS
%SSSSSSSSSSSSSSSSSSSSSSSSSSSSSSSSSSSSSSSSSSSSSSSSSSSSSSSSSSSSSSSSSSSSSSSSSSSSSSSSSSS
%SSSSSSSSSSSSSSSSSSSSSSSSSSSSSSSSSSSSSSSSSSSSSSSSSSSSSSSSSSSSSSSSSSSSSSSSSSSSSSSSSSS
%SSSSSSSSSSSSSSSSSSSSSSSSSSSSSSSSSSSSSSSSSSSSSSSSSSSSSSSSSSSSSSSSSSSSSSSSSSSSSSSSSSS

\section{Standard Proofs in the Theory of Primitive Roots }
Let $p$ be a large prime, let $x\geq p^{1/2+2\varepsilon}$ and let $\ord \chi$ be the multiplicative order of the character $\chi$. The indicator function of primitive roots is
\begin{equation}
	I(n)=\frac{\varphi(p-1)}{p-1}\sum_{d\mid p-1}\frac{\mu(d)}{\varphi(d)}\sum_{\ord \chi=d}\chi(n)=
	\begin{cases}
		1& n \text{ is a primitive root}\\
		0& n \text{ is not a primitive root.}
	\end{cases}
\end{equation}
The standard method used to prove the existence of primitive roots in $[1,x]$ is about a century old. Hundreds of papers have been published based on this method. It assumes that there are no primitive roots $n\leq x$. Then
\begin{align}\label{2}
	0&=\sum_{n\leq x} \frac{\varphi(p-1)}{p-1}\sum_{d\mid p-1}\frac{\mu(d)}{\varphi(d)}\sum_{\ord \chi=d}\chi(n)\\[.3cm]
	&=\sum_{n\leq x} \frac{\varphi(p-1)}{p-1}\,1+\sum_{n\leq x} \frac{\varphi(p-1)}{p-1}\sum_{1<d\mid p-1}\frac{\mu(d)}{\varphi(d)}\sum_{\ord \chi=d}\chi(n)\\[.3cm]
	&=M(x)+E(x).
\end{align}
The main term is
\begin{equation}
	M(x)=	\sum_{n\leq x} \frac{\varphi(p-1)}{p-1}1=\frac{\varphi(p-1)}{p-1}\,x+O(1)
\end{equation}
and the error term is
\begin{align}
	E(x)&=\sum_{n\leq x}\frac{\varphi(p-1)}{p-1}\sum_{1<d\mid p-1}\frac{\mu(d)}{\varphi(d)}\sum_{\ord \chi=d}\chi(n).
\end{align}

Going back to at least Vinogradov 1918, many leading mathematicians, have estimated the error like this or some variations of this routine:\\

1. Observe that the error term has a trivial evaluation
\begin{align}
	E(x)&=\sum_{n\leq x}\frac{\varphi(p-1)}{p-1}\sum_{1<d\mid p-1}\frac{\mu(d)}{\varphi(d)}\sum_{\ord \chi=d}\chi(n)=-\frac{\varphi(p-1)}{p-1}\,x.
\end{align}

However, switch the order of summation leads to a nontrivial upper bound.
\begin{align}
	E(x)&=\sum_{n\leq x}\frac{\varphi(p-1)}{p-1}\sum_{1<d\mid p-1}\frac{\mu(d)}{\varphi(d)}\sum_{\ord \chi=d}\chi(n)\\[.3cm]
	&=\frac{\varphi(p-1)}{p-1}\sum_{1<d\mid p-1}\frac{\mu(d)}{\varphi(d)}\sum_{\ord \chi=d}\sum_{n\leq x}\chi(n).
\end{align}
2. Take absolute value and apply some exponential sum estimate. 
\begin{align}
	|	E(x)|&\leq\frac{\varphi(p-1)}{p-1}\sum_{1<d\mid p-1}\Bigg|\frac{\mu(d)}{\varphi(d)}\sum_{\ord \chi=d}1\Bigg|\Bigg|\sum_{n\leq x}\chi(n)\Bigg|\\[.3cm]
	&\leq\frac{\varphi(p-1)}{p-1}\sum_{1<d\mid p-1}\Bigg|\frac{\mu(d)}{\varphi(d)}\Bigg|\Bigg|\sum_{\ord \chi=d}1\Bigg|\Bigg|\sum_{n\leq x}\chi(n)\Bigg|\\[.3cm]
	&\ll\frac{\varphi(p-1)}{p-1}\sum_{1<d\mid p-1}|\mu(d)|\Bigg|\sum_{n\leq x}\chi(n)\Bigg|\\[.3cm]
	&\ll\frac{\varphi(p-1)}{p-1}\cdot p^{\varepsilon}\cdot p^{1/2} \\[.3cm]
	&\ll p^{1/2+\varepsilon}		.
\end{align}
4. The inner sum in 
\begin{align}
\Bigg|\sum_{\ord \chi=d}1\Bigg|\Bigg|\sum_{n\leq x}\chi(n)\Bigg|\leq \varphi(d)\cdot A p^{1/2+\varepsilon},
\end{align}
where $A>0$ is a constant, depends on the character $\chi>1$ in the outer sum. But, the absolute values are computed separately.  \\

5. Since $x\geq p^{1/2+2\varepsilon}$, this implies a contradiction in \eqref{2}:
\begin{equation}
	0=M(x)+E(x)=\frac{\varphi(p-1)}{p-1}\,x+O(1)+O\left( p^{1/2+\varepsilon}\right). 
\end{equation}
%BBBBBBBBBBBBBBBBBBBBBBBBBBBBBBBBBBBBBBBBBBBBBBBBBBBBBB
%BBBBBBBBBBBBBBBBBBBBBBBBBBBBBBBBBBBBBBBBBBBBBBBBBBBBBB
%BBBBBBBBBBBBBBBBBBBBBBBBBBBBBBBBBBBBBBBBBBBBBBBBBBBBBB
%\newpage

\end{document}